\documentclass[a4paper,12pt]{article}
\usepackage{a4wide}
\usepackage[utf8]{inputenc}
\usepackage{amsmath,amsthm,amsfonts,latexsym,amssymb,bm,enumerate}

\usepackage[dvips]{graphicx}
\usepackage{psfrag}
\DeclareGraphicsExtensions{.eps,.art,.ART,.ps}

\DeclareFontFamily{OT1}{pzc}{}
\DeclareFontShape{OT1}{pzc}{m}{it}%
              {<-> s * pzcmi8t}{}
\DeclareMathAlphabet{\mathpzc}{OT1}{pzc}%
                                {m}{it}

\newtheorem{Thm}{Theorem}[section]

\newtheorem{Lem}[Thm]{Lemma}
\newtheorem{Cor}[Thm]{Corollary}
\newtheorem{Prop}[Thm]{Proposition}

\newtheorem{Rmk}[Thm]{Remark}

\newcommand{\R}{\mathbb{R}}

\newcommand{\eps}{\varepsilon}
\newcommand{\weak}{\rightharpoonup}
\newcommand{\h}{{\cal H}}
\newcommand{\C}{{\cal C}}

\newcommand{\D}{{\cal D}}

\newcommand{\rr}{{\cal R}}

\newcommand{\twzy}{(tw_*+y)}
\newcommand{\twyz}{(t_*w_*+y_*)}

\newcommand{\y}{y}

\newcommand{\cc}{c}

\newcommand{\w}{w}

\newcommand{\tG}{\tilde{G}}

\newcommand{\tz}{\hat{t}}
\newcommand{\sss}{s}

\newcommand{\cp}{c_{*,\lambda_1}^+}
\newcommand{\cm}{c_{*,\lambda_1}^-}

\newcommand{\csp}{c_*^+}
\newcommand{\csm}{c_*^-}
\newcommand{\usp}{u_*^+}
\newcommand{\usm}{u_*^-}
\newcommand{\dlmu}{(\Delta+\lambda_1)^{-1}}
\newcommand{\clp}{c_{\lambda_1}^+}
\newcommand{\clm}{c_{\lambda_1}^-}
\newcommand{\tsp}{t_*^+}
\newcommand{\tsm}{t_*^-}

\begin{document}
\begin{center}
{\Large\bf
Bifurcation curves of a diffusive logistic equation
with harvesting orthogonal to the first eigenfunction}\\
\ \\
Pedro Martins Girão\footnote{Email: pgirao@math.ist.utl.pt.
Partially supported by the Fundação para a Ciência e a Tecnologia (Portugal) and by project
UTAustin/MAT/0035/2008.} and
Mayte Pérez-Llanos\footnote{Email: mayte@math.ist.utl.pt.
Partially supported by the Fundação para a Ciência e a Tecnologia (Portugal) and by project
MTM2008-06326-C02-02.}
\\

\vspace{2.2mm}

{\small Center for Mathematical Analysis, Geometry and Dynamical Systems,\\ Mathematics Department,\\ Instituto Superior Técnico,\\ 1049-001 Lisbon, Portugal}

\end{center}

\begin{center}
{\bf Abstract}
\end{center}
\noindent
We study the global bifurcation curves of a diffusive logistic equation, when harvesting is orthogonal to the first eigenfunction
of the Laplacian, for values of the linear growth up to $\lambda_2+\delta$,
examining in detail their behavior as the linear growth rate crosses the first two eigenvalues.
We observe some new behavior with regard to earlier works concerning this equation.
Namely, the bifurcation curves suffer a transformation at $\lambda_1$, they are
compact above $\lambda_1$, there are precisely two families of degenerate solutions with Morse index equal to zero,
and the whole set of solutions below $\lambda_2$ is not a two dimensional manifold.

\section{Introduction}

This paper concerns the study of logistic equations of the form
\begin{equation}\label{a}
-\Delta u=au-f(u)-ch,
\end{equation}
in a smooth bounded domain $\Omega\in\R^N$, with $N\geq 1$. We are interested in weak solutions belonging to the space
$$
\h=\{u\in W^{2,p}(\Omega): u=0 \; \rm{on}\; \partial\Omega\},
$$
for some fixed $p>N$.
Let $\lambda_1$ and $\lambda_2$ be the first and second eigenvalues of the Dirichlet Laplacian
on $\Omega$, respectively.
We denote by $\phi$ the first eigenfunction
satisfying $\max_\Omega\phi=1$.
We assume that $\lambda_2$ is simple, with eigenspace spanned by $\psi$, and
we also normalize the second eigenfunction
so $\max_\Omega\psi=1$.

The competition term $f$ is assumed to satisfy the following hypotheses:
\begin{enumerate}[{\bf (a)}]
 \item[{\bf (i)}] $f\in C^2(\R)$.
 \item[{\bf (ii)}] $f(u)=0$ for $u\leq M$, and $f(u)>0$ for $u>M$; throughout $M\geq 0$ is fixed.
 \item[{\bf (iii)}] $f''(u)\geq 0$.
 \item[{\bf (iv)}] $\lim_{u\to+\infty} \frac{f(u)}{u}=+\infty$.
\end{enumerate}

In \cite{OSS1}, the authors obtained global bifurcation curves, of positive solutions to (\ref{a}), for values of the parameter $a$ in a right neighborhood
of $\lambda_1$, when $f(u)=u^2$ and $h$ is a positive function.

In \cite{PG}, the first author generalized the results of \cite{OSS1} to competition terms satisfying {\bf (i)}-{\bf (iv)},
and studied the bifurcation curves, of sign changing solutions, for $a$ up to $\lambda_2+\delta$, for some $\delta>0$. This was also done under the assumption
that $h$ was positive a.e.\ in $\Omega$, a hypothesis which was used in the proof, although, as noted in \cite{PG},
in a right neighborhood of $\lambda_1$, one may relax the requirement on $h$ to
$\int_\Omega h\phi\,dx\neq0$.

In this paper, we analyze the situation when the harvesting function $h$, which in our case might be more appropriately
called harvesting and plantation function,
is orthogonal to the first eigenfunction of the Laplacian.
The biological interpretation gives our context but should be taken with care, as it breaks down in several circumstances,
for instance, if the solutions become negative.
Our motivation is mathematical, we are forced to provide new arguments, and we suspected the geometry of
the problem would be different from the one in \cite{PG}. Indeed, it turns out that the bifurcation curves suffer a complete transformation when
the parameter $a$ crosses the first eigenvalue. We examine in detail the way in which this change occurs.
When seen in the $(a,u,c)$ space, the set of solutions of (\ref{a}) between $\lambda_1$ and $\lambda_2$ has the shape of a piece of a paraboloid,
with a flat bottom at $a=\lambda_1$. A 2-dimensional space of solutions is attached to this bottom at $a=\lambda_1$, along a segment,
and lies in the region $a\leq\lambda_1$. The whole set of solutions below $\lambda_2$ is not a two-dimensional manifold. Therefore we find a
richer behavior regarding this equation than in the earlier works.
Also, in contrast to the bifurcation curves obtained in the previous papers, our
curves turn out to be compact above $\lambda_1$, and, instead of one,
we get two families of degenerate solutions with Morse index equal to zero above $\lambda_1$.

Specifically, we assume:
\begin{enumerate}[{\bf (a)}]
\item[{\bf (a)}] $h\in L^\infty(\Omega)$.
\item[{\bf (b)}] $\int_\Omega h\phi\,dx=0$.
\item[{\bf (c)}] $\int_\Omega h\psi\,dx\neq 0$.
\end{enumerate}
Hypothesis {\bf (c)} also appears in \cite{PG}. Our main results are Theorems~\ref{conv}, \ref{stab},
\ref{convl1}, \ref{convl1b}, \ref{thmatl2} and \ref{thm}. The proofs involve bifurcation methods (\cite{CR,CR2}),
a blow up argument, the Morse indices,
and a careful choice of coordinates at each step.
In particular, around $\lambda_1$ we decompose the space $\h$ as in \cite{CP}.
In the end, we obtain a complete picture of the set of solutions for the parameter $a$ up to $\lambda_2+\delta$.

For other works related to logistic equations with harvesting we refer the reader to \cite{CDT,Eu,OSS2}.

This paper is organized as follows: We treat successively the cases where the linear growth parameter $a$ is equal to $\lambda_1$ (Section~2),
below $\lambda_1$ (Section~3), between $\lambda_1$ and $\lambda_2$ (Section~4), and greater than or equal to $\lambda_2$ (Section~5).

\vspace{\baselineskip}

\noindent {\bf Acknowledgments.} This research was conducted during a stay of the second author at IST.
She is grateful for the pleasant atmosphere.

\section{Linear growth $a$ equal to $\lambda_1$}

Assume $(\lambda_1,u,c)\in\R\times{\cal H}\times\R$ is a solution of (\ref{a}) with $a=\lambda_1$.
Multiplying both sides of (\ref{a}) by $\phi$ and integrating, taking into account
{\bf (b)} and $-\Delta\phi=\lambda_1\phi$, we deduce that $\int f(u)\phi\,dx=0$.
When the region of integration is omitted it is understood to be $\Omega$.
Because $\phi$ is positive in $\Omega$
and $f(u)$ is continuous in $\Omega$, we get that $f(u)\equiv 0$. This means $u\leq M$ by {\bf (ii)}.
Therefore,
for $a=\lambda_1$ the solutions of (\ref{a}) are those of the linear problem
$$
 -\Delta u=\lambda_1 u-ch,
$$
i.e.\ are of the form $(\lambda_1,u,c)$, where
$
u=t\phi+c\dlmu h,
$
with
\begin{equation}\label{Lambda}
(t,c)\in\Lambda:=\{(t,c)\in\R^2:t\phi+c\dlmu h\leq M\}.
\end{equation}
Thus, there is a bijection between the set of solutions of (\ref{a}) for $a=\lambda_1$ and $\Lambda$,
given by $(\lambda_1,t\phi+c(\Delta+\lambda_1)^{-1}h,c)\leftrightarrow(t,c)$.
The set $\Lambda$ is closed and convex.

Let
\begin{equation}\label{T}
T:=\sup\{t:{\rm there\ exists}\ c\ {\rm such\ that}\ (t,c)\in\Lambda\}.
\end{equation}
Taking $c=0$ and using the normalization $\max_\Omega\phi=1$, $(M,0)\in\Lambda$ and so $T\geq M$. The value of $T$ is finite.
Indeed, $h$ and $(\Delta+\lambda_1)^{-1}h$ are orthogonal to $\phi$ and so $(\Delta+\lambda_1)^{-1}h$ changes sign.
Let $\Omega_+$ be the set where $(\Delta+\lambda_1)^{-1}h$ is positive and
let $\Omega_-$ be the set where $(\Delta+\lambda_1)^{-1}h$ is negative. Suppose that $c\geq 0$; then
$M-c(\Delta+\lambda_1)^{-1}h\leq M$ on $\Omega_+$, and so if $(t,c)\in\Lambda$, then $t\phi\leq M$ on $\Omega_+$. Suppose $c<0$; then
$M-c(\Delta+\lambda_1)^{-1}h\leq M$ on $\Omega_-$, and so if $(t,c)\in\Lambda$, then $t\phi\leq M$ on $\Omega_-$.
In any case, $c\geq 0$ or $c<0$, $t\phi\leq M$ either on $\Omega_+$ or on $\Omega_-$.
We conclude that $T<+\infty$ as asserted.
The value $T$ is a maximum.

To characterize parts of the boundary of $\Lambda$,
we define two functions, $\clm$ and $\clp$, in the interval $]-\infty,T]$, by
\begin{equation}\label{clambda1}
\clm(t):=\min_{(t,c)\in\Lambda}c\qquad{\rm and}\qquad\clp(t):=\max_{(t,c)\in\Lambda}c.
\end{equation}
Clearly, $c_{\lambda_1}^-(T)\leq c_{\lambda_1}^+(T)$.
Notice that
\begin{equation}\label{infinito}
 \lim_{t\to-\infty}\clm(t)=-\infty\qquad {\rm and}\qquad\lim_{t\to-\infty}\clp(t)=+\infty,
\end{equation}
since, when $t\to-\infty$, denoting by $\nu$ the unit outward normal to $\Omega$, using Hopf's Lemma, we have
$\frac{\partial}{\partial\nu}(M-t\phi)=-t\frac{\partial\phi}{\partial\nu}\geq
-t\max_{\partial\Omega}\frac{\partial\phi}{\partial\nu}\to-\infty$, and
$M-t\phi\to+\infty$ uniformly in each compact subset of $\Omega$.
Therefore, because $(\Delta+\lambda_1)^{-1}h$ belongs to $C^1(\overline\Omega)$, as $t$ goes to $-\infty$, it is possible to guarantee that
$c(\Delta+\lambda_1)^{-1}h\leq M-t\phi$ for larger and larger values of $|c|$. This establishes (\ref{infinito}).
From (\ref{infinito}) and the fact that $\Lambda$ is convex,
it follows that $\clm$ is convex, continuous and strictly increasing.
Similarly, $\clp$ is concave, continuous and strictly decreasing.

If $M=0$, then $T=0$ and $\clm(0)=\clp(0)=0$, because the function $\dlmu h$ changes sign.
One can check with specific examples, that it might happen that $T=M$.
In such a situation $\clm(0)\leq 0\leq\clp(0)$ because $(M,0)\in\Lambda$ and so
$0\geq c_{\lambda_1}^-(M)\geq c_{\lambda_1}^-(0)$,
$0\leq c_{\lambda_1}^+(M)\leq c_{\lambda_1}^+(0)$ (see Figure~\ref{fig10}).
On the other hand, if $T>M$, then either $\clm(M)=0$ or $\clp(M)=0$ (see Figure~\ref{fig11}).
Indeed, take $M<t<T$ and $(t,c)\in\Lambda$. We have that
$c\dlmu h\leq M-t\phi$. The function $M-t\phi$ is negative
in an open subset of $\Omega$. Therefore,
$\{c:(t,c)\in\Lambda\}\subset\R^-$ or $\{c:(t,c)\in\Lambda\}\subset\R^+$, as the function
$c\dlmu h$ cannot vanish at any point of that open subset of $\Omega$.
This shows that $\clm(t)>0$ or $\clp(t)<0$, for $M<t<T$.
Passing to the limit as $t\searrow M$, $\clm(M)\geq 0$ or $\clp(M)\leq 0$. On the other hand,
since $(M,0)\in\Lambda$, it holds that $\clm(M)\leq 0\leq \clp(M)$. Therefore,
$\clm(M)=0$ or $\clp(M)=0$, as claimed.
In any of the three possible cases, $M=T=0$, $0<M=T$ and $0<M<T$, we have that
\begin{equation}\label{cc}
\cm:=\clm(0)\leq 0\qquad{\rm and}\qquad
\cp:=\clp(0)\geq 0.
\end{equation}

\begin{figure}
\centering
\begin{psfrags}
\psfrag{c}{{\tiny $c$}}
\psfrag{t}{{\tiny $t$}}
\psfrag{T}{{\tiny $T$}}
\psfrag{P}{{\tiny $M=0$}}
\psfrag{Q}{{\tiny $T=M$}}
\includegraphics[scale=.5]{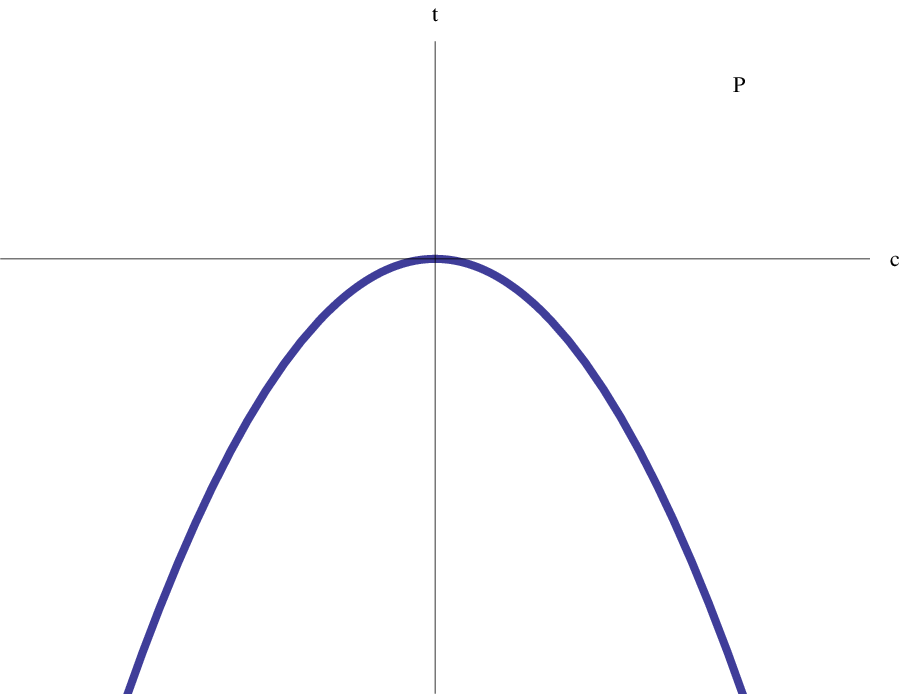}\ \ \ \ \ \ \ \ \ \ \ \ \ \ 
\includegraphics[scale=.5]{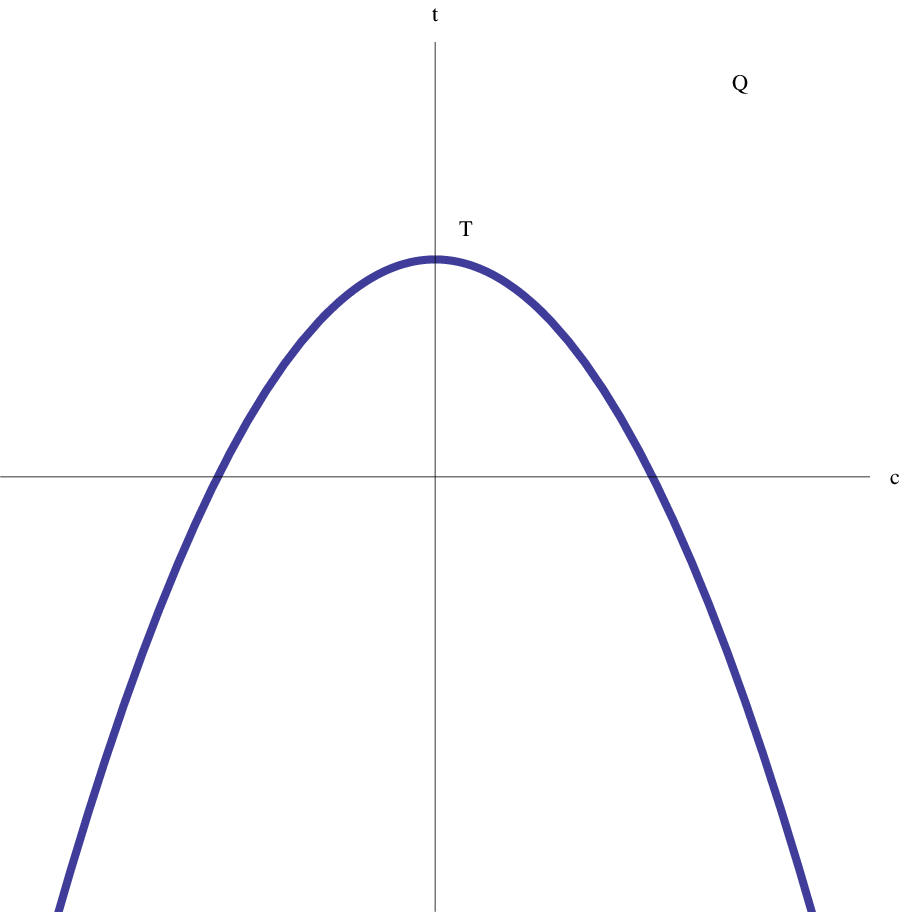}
\end{psfrags}
\caption{The boundary of the set $\Lambda$ when $M=0$ and when $T=M$.}\label{fig10}
\end{figure}

\begin{figure}
\centering
\begin{psfrags}
\psfrag{c}{{\tiny $c$}}
\psfrag{t}{{\tiny $t$}}
\psfrag{T}{{\tiny $\!\!\!T$}}
\psfrag{M}{{\tiny $\!\!\!M$}}
\psfrag{N}{{\tiny $\!M$}}
\psfrag{L}{{\tiny $c_{\lambda_1}^-(M)=0$}}
\psfrag{O}{{\tiny $c_{\lambda_1}^+(M)=0$}}
\includegraphics[scale=.5]{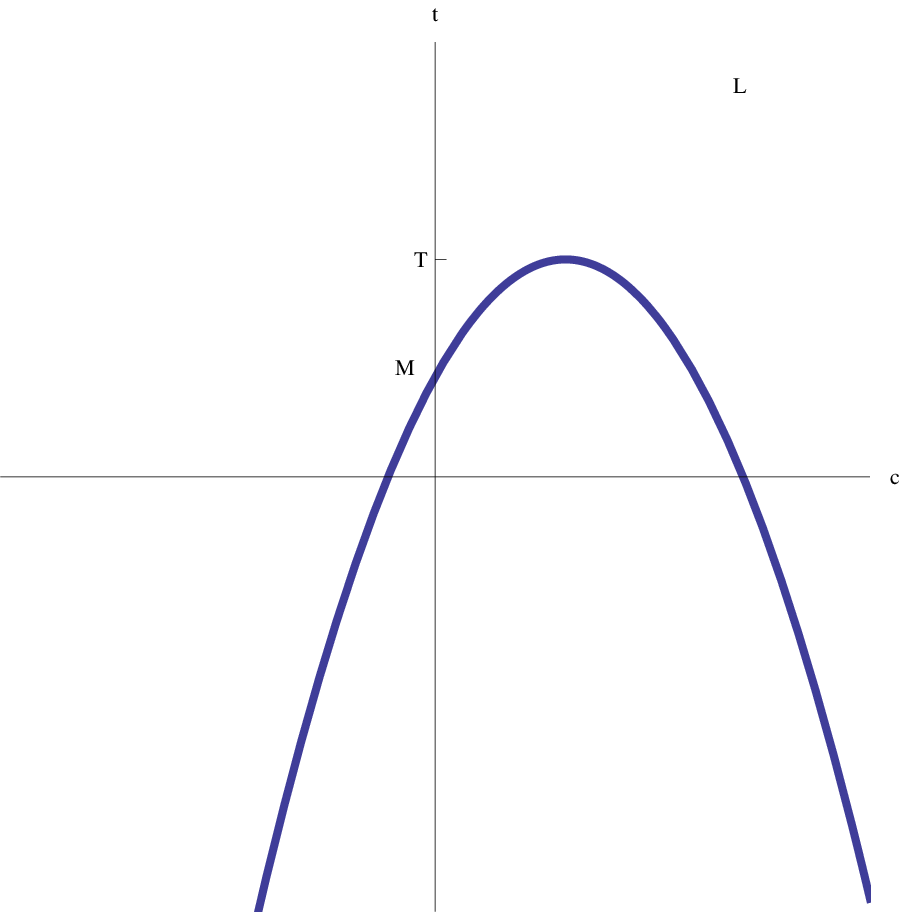}\ \ \ \ \ \ \ \ \ \ \ \ \ \ 
\includegraphics[scale=.5]{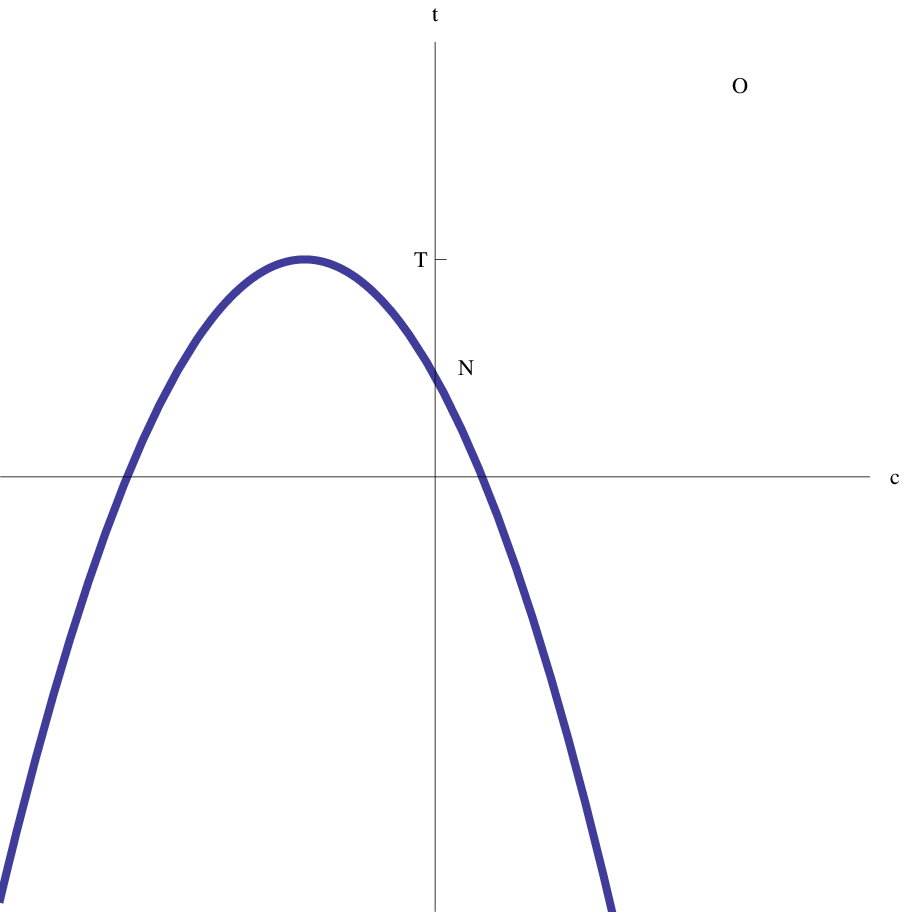}
\end{psfrags}
\caption{The boundary of the set $\Lambda$ when $T>M$.}\label{fig11}
\end{figure}

We can
describe the solutions $(a,u,c)$ of (\ref{a}) in a neighborhood of $(\lambda_1,t_0\phi+c_0\dlmu h,c_0)$ where, of course, $(t_0,c_0)\in\Lambda$.
We define
$$
\textstyle\rr:=\bigl\{y\in\h:\int y\phi\,dx=0\bigr\}.
$$
\begin{Lem}\label{TFI}
 Let $(t_0,c_0)\in\Lambda$ with $t_0\neq 0$. There exists a neighborhood $U\subset\R\times\h\times\R$
of $(\lambda_1,t_0\phi+c_0\dlmu h,c_0)$ such that the solutions of\/~{\rm (\ref{a})} in $U$ are a $C^1$ manifold
which can be parametrized, in a neighborhood $V$ of $(t_0,c_0)$, by
$(t,c)\mapsto(a(t,c),t\phi+y(t,c),c)$.
Here $y\in\rr$.
For $t>0$ we have that $a(t,c)\geq\lambda_1$, whereas
for $t<0$ it holds that $a(t,c)\leq\lambda_1$.
\end{Lem}
\begin{proof}
This lemma is a direct consequence of the Implicit Function Theorem applied to the function $g:\R\times\R\times\rr\times\R\to L^p(\Omega)$,
defined by
$g(a,t,y,c)=\Delta(t\phi+y)+a(t\phi+y)-f(t\phi+y)-ch$, at the point
$(\lambda_1,t_0,c_0(\Delta+\lambda_1)^{-1}h,c_0)\leftrightarrow(t_0,c_0)$ with $t_0\neq 0$. Let $(\alpha,z)\in \R\times{\cal R}$.
The derivative of $g$ with respect to $(a,y)$ in the direction of $(\alpha,z)$, $Dg_{(a,y)}(\alpha,z)$, computed at the point
$(\lambda_1,t_0,c_0(\Delta+\lambda_1)^{-1}h,c_0)$, is
\begin{eqnarray*}
g_a\alpha+g_yz&=&\alpha(t\phi+y)+\Delta z+az-f'(t\phi+y)z\\
&=&\alpha[t_0\phi+c_0(\Delta+\lambda_1)^{-1}h]+\Delta z+\lambda_1 z.
\end{eqnarray*}
We used the fact that $f'$ vanishes below $M$.
We check that this derivative is injective. Setting the previous derivative equal to zero, and looking at the $\phi$ component of the right hand side,
we first get that $\alpha=0$, because $(\Delta+\lambda_1)^{-1}h$ has no $\phi$ component and $t_0\neq 0$.
Thus $\Delta z+\lambda_1 z=0$. Since $z\in{\cal R}$, it follows that $z=0$.
This proves injectivity. It is easy to check that the derivative is also surjective.
So the derivative, computed at the point
$(\lambda_1,t_0,c_0(\Delta+\lambda_1)^{-1}h,c_0)$ with $t_0\neq 0$, is a homeomorphism from $\R\times{\cal R}$ to $L^p(\Omega)$.
The statement about the sign of $a$ follows from the equality
\begin{equation}\label{sinalt}
(a-\lambda_1)t\int\phi^2\,dx=\int f(u)\phi\,dx,
\end{equation}
where $u=t\phi+y$.
\end{proof}

For use in the next sections, where we consider values of $a$ different from $\lambda_1$, we make
the following
\begin{Rmk}\label{rmk}
 Let $(a_n,u_n,c_n)$ be a sequence of solutions of\/ {\rm (\ref{a})} with $(a_n)$ and $(c_n)$
bounded. Then $(u_n)$ is uniformly bounded above.
The same conclusion follows if, instead of assuming $(c_n)$ bounded,
we suppose that\/ $\bigl(\frac{c_n}{\max u_n}\bigr)$ is bounded.
\end{Rmk}
Indeed, denoting by $x_n$ a point of maximum for $u_n$, clearly
\begin{equation}\label{paracor}
a_nu_n(x_n)-f(u_n(x_n))-c_nh(x_n)\geq 0.
\end{equation}
Admitting that $u_n(x_n)\to +\infty$,
from
$$
\frac{f(u_n(x_n))}{u_n(x_n)}\leq a_n-\frac{c_n}{u_n(x_n)}h(x_n),
$$
whose left hand side is bounded, we contradict hypothesis {\bf (iv)}.

\section{Linear growth $a$ below $\lambda_1$}

In this section we are going to analyze the case $a<\lambda_1$. We observe that, for $c$ fixed, there exists a unique solution of (\ref{a}).
Indeed, to find the solution one just has to minimize the coercive (since $a<\lambda_1$ and $F(u):=\int_0^u f(s)\,ds$ is positive) functional
$$
\int\left[{\textstyle\frac{1}{2}}(|\nabla u|^2-au^2)+F(u)+chu\right]\,dx
$$
on $H^1_0(\Omega)$. By Remark~\ref{rmk}, if $(a,u,c)$ is a solution, ${\rm ess\,sup}\, u$ is finite. As $f(u)$ only depends on the positive part of $u$, $f(u)$ belongs to $L^\infty(\Omega)$.
By elliptic regularity theory (see \cite{GT}), $u$ belongs to ${\cal H}$.
The solution is nondegenerate. In fact, suppose that $v\in{\cal H}$ is such that
$$
-\Delta v-av+f'(u)v=0.
$$
Multiplying both sides of this equation by $v$ and integrating over $\Omega$, we get that
\begin{eqnarray*}
0&=&\int\left[|\nabla v|^2-av^2+f'(u)v^2\right]\,dx\\
&\geq&\int\left[|\nabla v|^2-av^2\right]\,dx\\
&\geq&(\lambda_1-a)\int v^2\,dx
\end{eqnarray*}
Thus $v=0$ and the solution is nondegenerate.
Thus, for a fixed $a$ the set of solutions of (\ref{a}) is a one dimensional $C^1$ manifold in $\{a\}\times\h\times\R$, that can be parametrized
by $c\mapsto (a,u_a(c),c)$. This follows from the Implicit Function Theorem applied to the function
$G:\R\times {\cal H}\times\R\to L^p(\Omega)$, defined by
$$
G(a,u,c)=\Delta u+au-f(u)-ch.
$$
Observe that, for $v\in{\cal H}$, the derivative $DG_{(u)}(v)$ of $G$ with respect to $u$ in the direction of $v$,
computed at a solution $(a,u,c)$, is $DG_{(u)}(v)=\Delta v+av-f'(u)v$.

For $a<\lambda_1$, the component in $\phi$ of a solution $(a,u,c)$,
$
\frac{\int u\phi\,dx}{\int\phi^2\,dx}=:t,
$
is nonpositive due to (\ref{sinalt}).
 We wish to examine the behavior of the solutions as $a$ increases to $\lambda_1$.

\begin{Lem}\label{proj}
Let $c^-$, $c^+$ be some constants such that
$c^-<\cm$ and $c^+>\cp$, with $\cm$, $\cp$ given in \eqref{cc}, and let $\tz <0$. There exists $\delta>0$ such that
for all $\lambda_1-\delta<a<\lambda_1$ and $(a,u,c)$ solution of\/ {\rm (\ref{a})}, with
$c^-\leq c\leq c^+$, we have that $t>\tau_{\tz}(c)$ with
\begin{equation}\label{tau}
 \tau_{\tz}(c):=\left\{\begin{array}{cl}
                  \min\bigl\{(\clm)^{-1}(c),\tz\bigr\}&{\rm if}\ c^-\leq c\leq \cm,\\
		  \tz&{\rm if}\ \cm\leq c\leq \cp,\\
		  \min\bigl\{\tz,(\clp)^{-1}(c)\bigr\}&{\rm if}\ \cp\leq c\leq c^+
                 \end{array}
	  \right.
\end{equation}
(see {\rm Figure~\ref{fig12}}).
\end{Lem}
\begin{figure}
\centering
\begin{psfrags}
\psfrag{c}{{\tiny $c$}}
\psfrag{t}{{\tiny $t$}}
\psfrag{a}{{\tiny $\!\!\!c^-$}}
\psfrag{b}{{\tiny $c^+$}}
\psfrag{s}{{\tiny $\tz$}}
\psfrag{u}{{\tiny $\tau_{\tz}(c)$}}
\includegraphics[scale=.6]{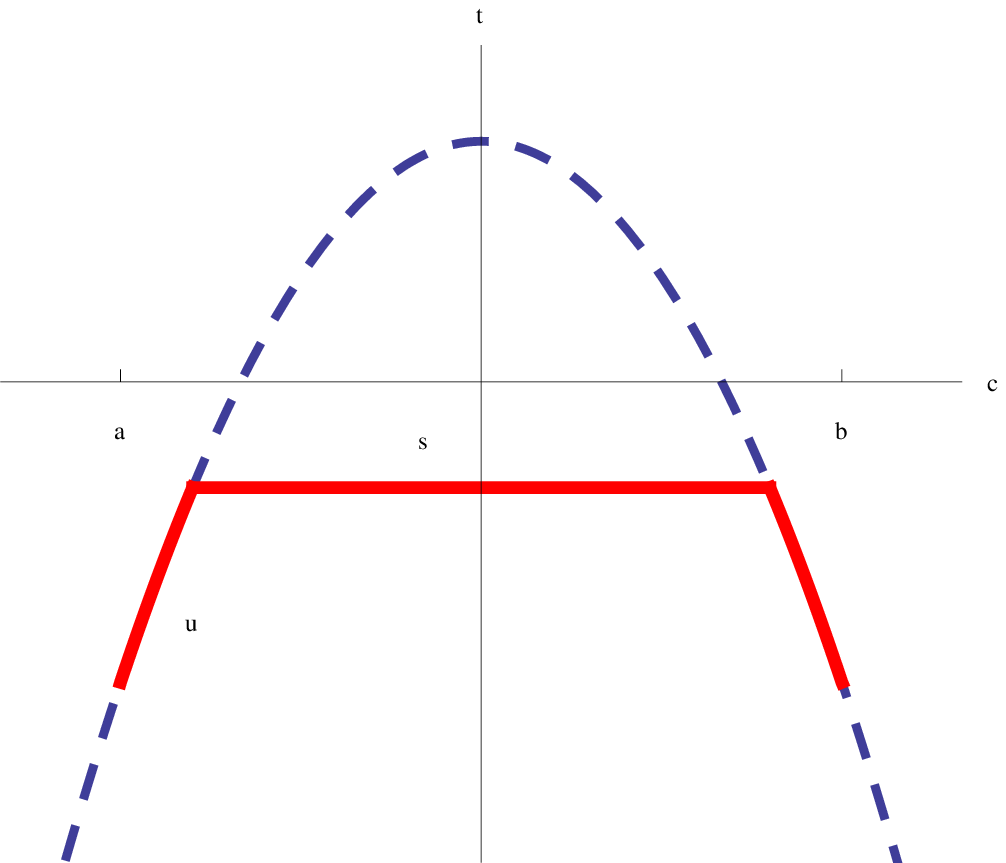}
\end{psfrags}
\caption{The graph of $\tau_{\tz}$.}\label{fig12}
\end{figure}
\begin{proof}
 We argue by contradiction. Let $a_n\nearrow\lambda_1$ be such that
$(a_n,u_n,c_n)$ is a solution of (\ref{a}), with $c^-\leq c_n\leq c^+$ and
$
t_n:=\frac{\int u_n\phi\,dx}{\int\phi^2\,dx}\leq\tau_{\tz}(c_n).
$
As remarked, for each fixed $a<\lambda_1$, the map $c\mapsto u_a(c)$ is, in particular, continuous, and $u_a(0)=0$. 
Note that $t$ is a continuous function of $u$, which itself is a continuous function of $c$.
Thus, taking into account $t(c_n)=t_n\leq\tau_{\tz}(c_n)<t(0)=0$ and using the Intermediate Value Theorem for each fixed $a$,
without loss of generality, by changing $c_n$ (and thus $u_n$), we may assume that
\begin{equation}\label{tneq}
t_n=\tau_{\tz}(c_n).
\end{equation}
 In addition, we may suppose that $c_n\to c_0$. This implies that $t_n\to t_0$.

Let $y_n=u_n-t_n\phi$.
Multiplying (\ref{a}) by $y_n$ and integrating over $\Omega$,
$$
 \int|\nabla y_n|^2\,dx=a_n\int y_n^2\,dx-\int f(t_n\phi+y_n)y_n\,dx-c_n\int hy_n\,dx.
$$
We observe that the second term in the right hand side is nonpositive since
\begin{equation}\label{padepois}
-\int f(t_n\phi+y_n)y_n\,dx=-\int f(t_n\phi+y_n)(t_n\phi+y_n)\,dx+t_n\int f(t_n\phi+y_n)\phi\,dx.
\end{equation}
Thus, we have that
$$
 \int|\nabla y_n|^2\,dx\leq a_n\int y_n^2\,dx-c_n\int hy_n\,dx,
$$
which, together with
$$
 \int|\nabla y_n|^2\,dx\geq \lambda_2\int y_n^2\,dx,
$$
implies first that $(y_n)$ is bounded in $L^2(\Omega)$ and then that $(y_n)$ is bounded in $H^1_0(\Omega)$.
We may assume that $y_n\to y_0$ in $L^2(\Omega)$ and a.e.\ in $\Omega$.
By Remark~\ref{rmk}, $({\rm ess\,sup}\,u_n)$ is uniformly bounded.
Letting $u_0=t_0\phi+y_0$, from the Dominated Convergence Theorem it follows that $f(u_n)\to f(u_0)$ in $L^p(\Omega)$.
Using equation (\ref{a}) and elliptic regularity theory (see \cite{GT}), $y_n\to y_0$ in $\h$.

From the previous paragraph,
the limit $(\lambda_1,u_0,c_0)$ satisfies equation (\ref{a}) and, from (\ref{tneq}), $t_0=\tau_{\tz}(c_0)$.
We use Lemma~\ref{TFI} at the point $(t_0,c_0)$. The solutions of (\ref{a}) in a neighborhood of the image of $(t_0,c_0)$
can be parametrized by
$(t,c)\mapsto(a(t,c),t\phi+y(t,c),c)$. From the choice of $(a_n,u_n,c_n)$,
we have that $a(\tau_{\tz}(c_n),c_n)=a(t_n,c_n)=a_n<\lambda_1$.

To obtain the desired contradiction, we show next that $a(\tau_{\tz}(c),c)=\lambda_1$ for any $c\in[c^-,c^+]$.
There are three possible cases: (i) $c=\clm(t)$ for some $t<\tz$,
(ii) $c=\clp(t)$ for some $t<\tz$ or (iii) $\clm(\tz)\leq c\leq \clp(\tz)$.
In case~(i) $\tau_{\tz}(c)=t$. So the solution
$(\lambda_1,t\phi+\clm(t)\dlmu h,\clm(t))$ can be written as $(\lambda_1,\tau_{\tz}(c)\phi+c\dlmu h,c)$, which means that $a(\tau_{\tz}(c),c)=\lambda_1$.
Similarly in case~(ii).
In case~(iii) $\tau_{\tz}(c)=\tz$. Since $\clm(\hat{t})\leq c\leq\clp(\hat{t})$,
$(\lambda_1,\hat{t}\phi+c\dlmu h,c)$ is a solution of (\ref{a}). Therefore $\lambda_1=a(\tz,c)=a(\tau_{\tz}(c),c)$.

In conclusion, on the one hand $a(\tau_{\tz}(c_n),c_n)<\lambda_1$ and
on the other hand $a(\tau_{\tz}(c),c)=\lambda_1$ for any $c\in[c^-,c^+]$.
We reached a contradiction. The lemma is proved.
\end{proof}

We extend the definition of $\tau$ in (\ref{tau}) to zero,
$$ 
 \tau_{0}(c):=\left\{\begin{array}{cl}
                  (\clm)^{-1}(c)&{\rm if}\ c^-\leq c\leq \cm,\\
		  0&{\rm if}\ \cm\leq c\leq \cp,\\
		  (\clp)^{-1}(c)&{\rm if}\ \cp\leq c\leq c^+.
                 \end{array}
	  \right.
$$ 
In the next theorem, we prove that, as $a\nearrow\lambda_1$, the solutions $(a,u_a(c),c)$ converge to
\begin{equation}\label{bo}
 (\lambda_1,u_{\lambda_1}(c),c):=(\lambda_1,\tau_0(c)\phi+c\dlmu h,c).
\end{equation}
\begin{Thm}\label{conv}
Let $c^-$, $c^+$ be some constants such that $-\infty<c^-<\cm\leq\cp<c^+<\infty$.
The solutions $(a,u_a(c),c)$ are such that the functions $c\mapsto u_a(c)$, for $c\in[c^-,c^+]$,
converge uniformly in $\h$ to $u_{\lambda_1}(c)$ as $a\nearrow\lambda_1$.
\end{Thm}
\begin{proof}
 Again the proof is by contradiction. We admit that there exists $\delta>0$, $a_n\nearrow\lambda_1$ and $c_n\in[c^-,c^+]$
such that the corresponding solutions $(a_n,u_{a_n}(c_n),c_n)$ satisfy
\begin{equation}\label{uniforme}
 \|u_{a_n}(c_n)-u_{\lambda_1}(c_n)\|_{\h}\geq\delta.
\end{equation}
We may assume that $c_n\to c_0$. Let $\tz<0$. By Lemma~\ref{proj}, for large $n$, $\tau_{\tz}(c_n)\leq t_n\leq 0$.
Modulo a subsequence, $t_n\to t_0$, where $\tau_{\tz}(c_0)\leq t_0$. Since $\tz$ is arbitrary,
$\tau_0(c_0)\leq t_0$. Again arguing as in the proof of Lemma~\ref{proj},
$y_n=u_{a_n}(c_n)-t_n\phi\to y_0$ in $\h$ and $(a_n,u_{a_n}(c_n),c_n)$ converge to
$(\lambda_1,u_0,c_0):=(\lambda_1,t_0\phi+y_0,c_0)$ which must be a solution of (\ref{a}). Using that $\tau_{0}(c_0)\leq t_0$ and $t_0\leq 0$,
we conclude $t_0=\tau_{0}(c_0)$, because there are no solutions of (\ref{a}) corresponding to a nonengative
$t_0$  satisfying $\tau_{0}(c_0)<t_0$. Therefore $(\lambda_1,u_0,c_0)=(\lambda_1,\tau_{0}(c_0)\phi+y_0,c_0)=
(\lambda_1,u_{\lambda_1}(c_0),c_0)$, according to (\ref{bo}) since necessarily $y_0=c_0(\Delta+\lambda_1)^{-1}h$.
This shows that $u_{a_n}(c_n)$ converges to $u_{\lambda_1}(c_0)$ in ${\cal H}$. Examining (\ref{bo}),
$c\mapsto u_{\lambda_1}(c)$ is continuous in ${\cal H}$.
Passing to the limit in both sides of (\ref{uniforme}), we get that
$$
0=\|u_{\lambda_1}(c_0)-u_{\lambda_1}(c_0)\|_{\h}\geq\delta.
$$
This contradiction proves the theorem.
\end{proof}
In Figure~\ref{fig15}, we plot the curves $c\mapsto(a,t(c),c)$ for $a$ in an interval $]\lambda_1-\delta,\lambda_1[$, for some
small $\delta>0$. The value of $a$ increases in the vertical direction.
\begin{figure}
\centering
\begin{psfrags}
\psfrag{c}{{\tiny $c$}}
\psfrag{a}{{\tiny $a$}}
\psfrag{t}{{\tiny $t$}}
\includegraphics[scale=.6]{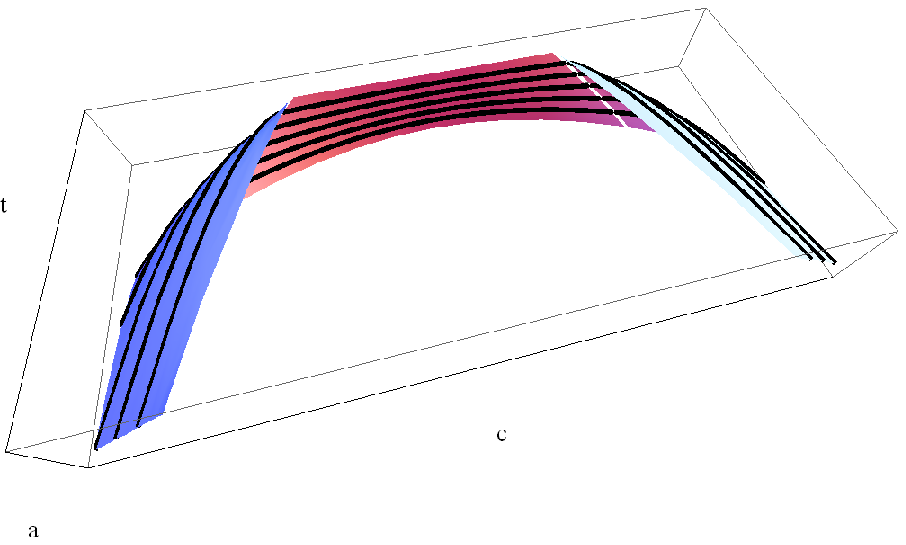}\ \ \ \ \ \ \ \
\includegraphics[scale=.6]{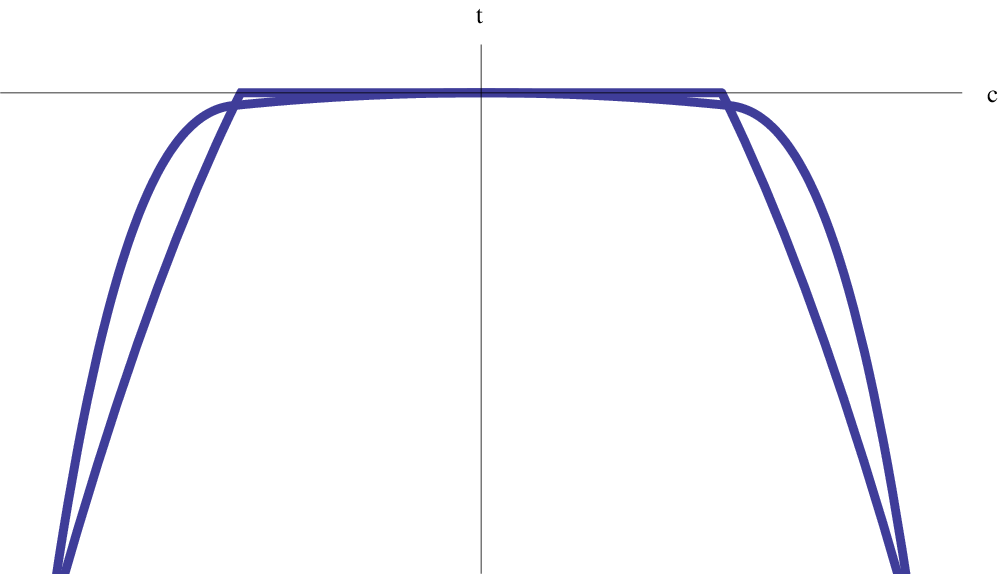}
\end{psfrags}
\caption{In the left image one can see the surface formed by the curves $c\mapsto(a,t(c),c)$ for $a$ in an interval $]\lambda_1-\delta,\lambda_1[$.
In the right image one can see a curve $c\mapsto(a,t(c),c)$ for an $a$ in $]\lambda_1-\delta,\lambda_1[$ and one can see
the boundary of the set $\Lambda\cap\{(t,c):t\leq 0\}$. This figure corresponds to a case where $M>0$.}\label{fig15}
\end{figure}

\section{Linear growth $a$ between $\lambda_1$ and $\lambda_2$}

In the case $c=0$, equation (\ref{a}) reduces to
\begin{equation}\label{b}
-\Delta u=au-f(u).
\end{equation}
From \cite[Theorem~2.1]{PG}, we know
\begin{Lem}\label{thmcz}
Suppose that $f$ satisfies {\rm\bf (i)}-{\rm\bf (iv)}.
The set of positive solutions $(a,u)$ of {\rm (\ref{b})} is a connected one dimensional manifold
$\C_\dagger$ of class $C^1$ in $\R\times\h$.
The manifold is the union of the
segment $\{(\lambda_1,t\phi):t\in\,]0,M]\}$ with a graph $\{(a,u_\dagger(a)):a\in\,]\lambda_1,+\infty[\}$.
The solutions are strictly increasing along $\C_\dagger$.
For $a>\lambda_1$ every positive solution is stable and, at each $a$, equation {\rm (\ref{b})} has no other stable solution besides $u_\dagger(a)$.
\end{Lem}
We turn to the case $c\neq 0$. Recall that by Remark~\ref{rmk}, for $a$ and $c$ bounded, solutions
$(a,u,c)$ are such that $u$ is bounded above.
On the other hand, for $c$ unbounded, we have
\begin{Lem}\label{ces}
 Let $I$ be any compact interval in $\R$. There exists $K>0$ such that, for all $(a,u,c)$ solution of\/ {\rm (\ref{a})}
with $a\in I$ and $|c|>K$, it holds $u\leq|c|$.
\end{Lem}
\begin{proof}
 By contradiction, suppose that $(a_n,u_n,c_n)$, with $a_n\in I$ and $|c_n|\to\infty$, is a sequence of solutions of (\ref{a}),
verifying  that$\max u_n>|c_n|$. From Remark~\ref{rmk},
 $(u_n)$ is uniformly bounded above, which is a contradiction.
\end{proof}
Taking into account inequality (\ref{paracor}), we immediately obtain
\begin{Cor}\label{fc}
 Under the conditions of the previous lemma, there exists a constant $C$ such that $f(u)\leq C|c|$.
\end{Cor}

In the following proposition we show that for large values of $|c|$ equation (\ref{a}) has no solutions.
\begin{Prop}\label{ccota}
 Let $\lambda_2<a_2<\lambda_3$ and $J=]\lambda_1,a_2]$. There exists $\overline c >0$ such that for all $a\in J$
and $(a,u,c)$ solution of\/ {\rm (\ref{a})}, we have that $|c|\leq \overline c$.
\end{Prop}
\begin{proof}
 We argue by contradiction. Suppose that $(a_n,u_n,c_n)$ is a solution to~(\ref{a}) with $a_n\in J$, and $c_n\to+\infty$
or $c_n\to -\infty$.
We define $\sss=+1$ in the first case and $\sss=-1$ in the second case.
 Without loss of generality, we assume that $a_n\to a$. Define $v_n=\frac{u_n}{\sss c_n}$. The function $v_n$ satisfies
\begin{equation}\label{eq_vn}
\Delta v_n+a_n v_n-\frac{f(u_n)}{\sss c_n}-\sss h=0.
\end{equation}
From Corollary~\ref{fc}, we know that
\begin{equation}\label{eq_fc}
\frac{f(u_n)}{\sss c_n}\leq C
\end{equation}
for sufficiently large $n$. Since $f$ is nonnegative, we also have $\frac{f(u_n)}{\sss c_n}\geq 0$.
Recall that we assumed that $\lambda_2$ is simple
and decompose
$$
v_n=t_n\phi+\eta_n\psi+w_n,
$$
where $w_n$ denotes the component of $v_n$ orthogonal to both $\phi$ and $\psi$.
We prove successively that $w_n$ is uniformly bounded in $L^2(\Omega)$ and that $t_n$, $\eta_n$ are bounded.
By Corollary~\ref{fc}, the function $\frac{f(u_n)}{\sss c_n}+\sss h$ is bounded in $L^\infty(\Omega)$. Thus,
its component $z_n$ orthogonal to the first two eigenfunctions is also bounded in $L^\infty(\Omega)$. Since $a_n$ is bounded away from $\lambda_3$,
$(w_n)$ is uniformly bounded in $L^2(\Omega)$.
This is because
$$
\Delta w_n+a_nw_n=z_n,\quad{\rm or}\quad w_n=(\Delta+a_n)^{-1}z_n.
$$
So, for example, the component of $w_n$ on a third eigenfunction is the component of $z_n$ on that third eigenfunction divided by $(a_n-\lambda_3)$.
Similarly to (\ref{sinalt}), the values $t_n$ are given by
$$
(a_n-\lambda_1)t_n\int \phi^2\,dx=\int{\textstyle \frac{f(u_n)}{\sss c_n}}\phi\,dx
$$
and hence are nonnegative, $t_n\geq 0$. An upper bound for $t_n$ follows from $v_n\leq 1$,
due to Lemma~\ref{ces},
which gives
$t_n\leq\int\phi\,dx\left/\int\phi^2\,dx\right.$. Suppose that $|\eta_n|\to\infty$. The sequence $\eta_n\psi$ is not bounded above or
below in $L^\infty(\Omega)$ because $\psi$ changes sign. This contradicts $v_n\leq 1$
(because $v_n^+$ and $w_n$ are bounded in $L^2(\Omega)$, $(\eta_n\psi)^+$ is unbounded in $L^2(\Omega)$ and $t_n\phi\geq 0$). We conclude $(\eta_n)$ is bounded and
$(v_n)$ is uniformly bounded in $L^2(\Omega)$.

From (\ref{eq_vn}), $(v_n)$ is uniformly bounded in $H^1_0(\Omega)$. We may assume that
$v_n\weak v$ in $H^1_0(\Omega)$, $v_n\to v$ in $L^2(\Omega)$ and $v_n\to v$ a.e.\ in $\Omega$.
We claim that $v\leq 0$ a.e.\ in $\Omega$. Suppose that there exists a point $x\in\Omega$,
in the set where $v_n\to v$, such that $v(x)>0$.
Then $u_n(x)\to +\infty$. So,
$$
\frac{f(u_n(x))}{\sss c_n}=\frac{f(u_n(x))}{u_n(x)}\frac{u_n(x)}{\sss c_n}\to+\infty\times v(x)=+\infty.
$$
This contradicts (\ref{eq_fc}) a.e.\ and shows that $v\leq 0$ a.e.\ in $\Omega$.
In order to pass to the limit in (\ref{eq_vn}), we observe that
$\frac{f(u_n)}{\sss c_n}\weak f_\infty$ in $L^2(\Omega)$,
where $f_\infty\geq 0$. The limit equation is
$$
\Delta v+av-f_\infty-\sss h=0.
$$
Multiplying both sides by $\phi$ and integrating over $\Omega$, we arrive to
$$
(a-\lambda_1)\int v\phi\,dx=\int f_\infty\phi\,dx.
$$
The left hand side is nonpositive and the right hand side is nonnegative.
This implies that $v$ and $f_\infty$ are both identically equal to zero.
We have reached a contradiction because $h$ is nontrivial.
\end{proof}
The conclusion of the previous proposition also holds for the case where $\lambda_2$ has
multiplicity greater than one, since if a linear combination of second
eigenfunctions is bounded, then each of the coefficients of that linear combination
is bounded.

In the next lemma, we give a condition which ensures that a
sequence of solutions of (\ref{a}) converges, modulo a subsequence.
\begin{Lem}\label{continua}
Let $(a_n,u_n,c_n)$ be a sequence of solutions of\/ {\rm (\ref{a})} with $a_n\in J$, $a_n\to a$ and $c_n\to c$.
Then, modulo a subsequence, $(u_n)$ converges in $\h$.
\end{Lem}
\begin{proof}
 Using an argument similar to, but simpler than, the one in the proof of the previous proposition, one can show  that $(u_n)$
is uniformly bounded in $H_0^1(\Omega)$. By elliptic regularity theory, $(u_n)$ is uniformly bounded in $\h$.
Modulo a subsequence, $u_n\weak u$ in $H^1_0(\Omega)$, $u_n\to u$ in $L^2(\Omega)$ and $u_n\to u$ a.e.\ in $\Omega$.
Working with that subsequence and
subtracting the equations for $u_m$ and $u_n$, one can prove that $(u_n)$ converges in $\h$.
\end{proof}
Let $a>\lambda_1$. Starting at the stable solution $(a,u_\dagger(a),0)$,
and keeping $a$ fixed, we can use the Implicit Function Theorem to follow
a branch of solutions, taking $c$ as parameter. Lemma~\ref{continua} guarantees that the branch will not go to infinity.
Since, from Proposition~\ref{ccota}, solutions do not exist for large $|c|$, there must exist at least two degenerate solutions
with Morse index equal to zero,
$(a,\usm(a),\csm(a))$ and $(a,\usp(a),\csp(a))$, the first corresponding to a negative value of $c$ and the
second corresponding to a positive value of $c$.
 We recall that the Morse index of a solution is the number of negative eigenvalues
of the linearized problem at the solution, and we recall that the solution is said to be degenerate if one of the eigenvalues of
the linearized problem is equal to zero. In the next lemma we examine the behavior of the
branch of solutions around a degenerate solution with
Morse index equal to zero.
\begin{Lem}[{\cite[Theorem~3.2]{CR2}}, {\cite[p.~3613]{OSS1}}]\label{turn} Let $a>\lambda_1$ be fixed and
${\bm p}_*=(u_*,c_*)$ be a degenerate solution with Morse index equal to zero, with $c_*>0$ (respectively $c_*<0$).
There exists a neighborhood of ${\bm p}_*$ in $\h\times\R$ such that the
set of solutions of {\rm (\ref{a})} in the neighborhood is a $C^1$ manifold.
This manifold is ${\bm m}^\sharp\cup\{{\bm p}_*\}\cup{\bm m}^*$. Here
\begin{itemize}
\item
${\bm m}^\sharp$ is a manifold of nondegenerate solutions with Morse index equal to one, which
is a graph $\{(u^\sharp(c),c):c\in\,]c_*-\eps_*,c_*[\}$ ($\{(u^\sharp(c),c):c\in\,]c_*,c_*+\eps_*[\}$).
\item ${\bm m}^*$ is a manifold of stable solutions, which
is a graph $\{(u^*(c),c):c\in\,]c_*-\eps_*,c_*[\}$ ($\{(u^*(c),c):c\in\,]c_*,c_*+\eps_*[\}$).
\end{itemize}
The value $\eps_*$ is positive. The manifolds ${\bm m}^\sharp$ and ${\bm m}^*$ are connected by $\{{\bm p}_*\}$.
\end{Lem}
\begin{proof} 
Let $(u_*,c_*)$ be a degenerate solution with Morse index equal to zero.
Let
$t_*$ and $y_*$ be such that
$u_*=t_*w_*+y_*$, with $$w_*\in S:=\{w\in\h:\textstyle\int w^2\,dx=\int\phi^2\,dx\}$$ satisfying
\begin{equation}\label{linearized}
\Delta w_*+aw_*-f'(u_*)w_*=0.
\end{equation}
$w_*>0$, and $$y_*\in\rr_{w_*}=\textstyle\{\omega\in\h:\int\omega w_*\,dx=0\}.$$

We can assume that 
the first eigenfunction $w_*$ is non negative because it minimizes the Rayleigh quotient.
By Remark~2.2 and elliptic regularity theory, $w_*$ is $C^2(\Omega)$.
Suppose that $w_*$ is zero at a point in the interior of $\Omega$.
As $w_*$ has to be positive somewhere, centering a ball at a point where $w_*$ is positive and enlarging it,
one can pick a ball such that $w_*$ is positive in the interior of the ball and zero somewhere on
the boundary of the ball, say $x_0$, with $x_0$ in the interior of $\Omega$.
Applying Lemma 3.4 of \cite{GT}
(note the last assertion, if $w_*(x_0)=0$ the same conclusion holds irrespective of the
sign of the function $c$ in \cite{GT}), we obtain that $\frac{\partial w_*}{\partial \nu}(x_0)<0$ and this leads to $w_*$ negative along the normal direction
to the boundary of the ball, in the exterior of the ball. This contradicts that $w_*$ is nonengative. Hence $w_*$ is positive in $\Omega$.

Combining (\ref{a}) with (\ref{linearized}), we obtain that
$$
\int(f'(u_*)u_*-f(u_*))w_*\,dx=c_*\int hw_*\,dx.
$$
Observe that
\begin{equation}\label{vanish}
 \int h w_*\,dx\neq 0.
\end{equation}
Indeed,
otherwise $f'(u_*)u_*-f(u_*)=0$. This implies that $u_*\leq M$, because of hypotheses {\bf (i)}$-${\bf (iii)} on the function $f$.
In fact, by {\bf (ii)} and {\bf (iii)}, $u\mapsto f'(u)u-f(u)$ is increasing for $u\geq M$. Since the solution of the ordinary
differential equation $f'(u)u-f(u)=0$ is $f(u)=cu$ and $f$ is continuous, $f'(u)u-f(u)$ cannot be zero for $u>M$ unless $M=0$.
But then $f$ would not be differentiable at zero, contradicting {\bf (i)}. So $f'(u_*)u_*-f(u_*)=0$ implies that $u_*\leq M$, as stated above. In this situation
the term $f'(u_*)w_*$ in equation (\ref{linearized}) would vanish and so $a=\lambda_1$,
contrary to our assumption. Therefore, if $c_*>0$, then $\int hw_*\,dx>0$, and if
$c_*<0$, then $\int hw_*\,dx<0$.

We let $\tG:\R\times\rr_{w_*}\times\R\times S\times\R\to L^p(\Omega)\times L^p(\Omega)$ be defined by
\begin{eqnarray*}
\tG(t,y,c,w,\mu)&=&(\Delta\twzy+a\twzy-f\twzy-ch,\\ &&\ \Delta w+aw-f'\twzy w+\mu w).
\end{eqnarray*}
We have that $\tG(t_*,y_*,c_*,w_*,0)=0$. We may use the Implicit Function Theorem to describe the solutions of $\tG=0$
in a neighborhood of $(t_*,y_*,c_*,w_*,0)$. Indeed, at this point,
\begin{eqnarray*}
\tG_yz+\tG_c\gamma+\tG_w\omega+\tG_\mu\nu 
&=&(\Delta z+az-f'\twyz z-\gamma h,\\
&&\ \Delta\omega+a\omega-f'\twyz\omega\\ &&\ \ -f''\twyz zw_*+\nu w_*).
\end{eqnarray*}
If this derivative vanishes, then we get that $\gamma=0$ (multiply by $w_*$ and integrate) and then $z=0$ (because $z\in\rr_{w_*}$ and the first eigenvalue is simple).
Since $z=0$, this implies (using the same argument) that $\nu=0$ and then $\omega=0$ (because $\omega\in\rr_{w_*}$).
The derivative is a homeomorphism from $\rr_{w_*}\times\R\times\rr_{w_*}\times\R$ to $L^p(\Omega)\times L^p(\Omega)$.
So the solutions of $\tG=0$ in a neighborhood of $(t_*,y_*,c_*,w_*,0)$ lie
on a curve $t\mapsto(t,\y(t),\cc(t),\w(t),\mu(t))$. Differentiating $\tG(t,\y(t),\cc(t),\w(t),\mu(t))=(0,0)$
once with respect to $t$,
$$
\Delta z+az-f'\twyz z-\gamma h=-\left(\Delta w_*+aw_*-f'\twyz w_*\right)=0,
$$
$$
\Delta\omega+a\omega-f'\twyz\omega-f''\twyz zw_*+\nu w_*=f''\twyz w_*^2
$$
Here $z=\y'(t_*)$, $\gamma=\cc'(t_*)$, $\omega=\w'(t_*)$ and $\nu=\mu'(t_*)$. Clearly, both $\gamma$ and $z$ vanish.
This implies that
$$
\nu=\frac{\int f''\twyz w_*^3\,dx}{\int w_*^2\,dx}.
$$
As we just saw,
it is impossible for $\max_\Omega\twyz\leq M$. 
Hence,
\begin{equation}\label{nu}
\mu'(t_*)>0.
\end{equation}
Differentiating the first equation in $\tG(t,\y(t),\cc(t),\w(t),\mu(t))=(0,0)$
twice with respect to $t$, at $t_*$,
\begin{eqnarray*}
\Delta z'+az'-f'\twyz z'-\gamma' h-f''\twyz z(w_*+z)=\qquad\\
\qquad\qquad\qquad\qquad\qquad\qquad f''\twyz w_*(w_*+z).
\end{eqnarray*}
Here $z=\y'$.
This can be rewritten as
\begin{eqnarray*}
\Delta z'+az'-f'\twyz z'-\gamma' h&=&f''\twyz (w_*+z)^2\\ &=&f''\twyz w_*^2,
\end{eqnarray*}
as $z(t_*)=0$. Multiplying by $w_*$ and integrating,
$$
\cc''(t_*)=-\,\frac{\int f''\twyz w_*^3\,dx}{\int hw_*\,dx}.
$$
This is formula (2.7) of \cite{OSS1}.
So $\cc''(t_*)$ is negative if $c_*>0$ and positive if $c_*<0$.
Suppose that $c_*>0$ (respectively, $c_*<0$).
As $t$ increases from $t_*$,
$\cc(t)$ decreases (respectively, increases) and the solution becomes stable.
So the ``end" of ${\bm m}^*$ coincides with the piece of curve parametrized by
$t\mapsto(c(t),t\w(t)+\y(t))$, for $t$ in a right neighborhood of $t_*$.
A parametrization of ${\bm m}^\sharp$ is obtained by taking $t$ in a left neighborhood of $t_*$.
\end{proof}
The next proposition guarantees that the degenerate solutions vary smoothly with $a$.
\begin{Prop}\label{thmd}
The set of degenerate solutions $(a,u,c)$ of {\rm (\ref{a})} with Morse index equal to zero
in $]\lambda_1,+\infty[\times\h\times\R$
is the disjoint union of two connected one dimensional manifolds $\D_*^-$ and $\D_*^+$ of class $C^1$.
Each manifold is a graph $\{(a,u_*^-(a),c_*^-(a)):a\in\,]\lambda_1,+\infty[\}$ and
$\{(a,u_*^+(a),c_*^+(a)):a\in\,]\lambda_1,+\infty[\}$.
\end{Prop}
\begin{proof}[Sketch of the proof]
To prove that the degenerate solutions can be followed using the parameter $a$, we
apply (\ref{vanish}) and the argument in the proof of Theorem~3.1 in \cite{PG}. On the other hand,
suppose that there were more than the two degenerate solutions with Morse index equal to zero,
$(a,\usm(a),\csm(a))$ and $(a,\usp(a),\csp(a))$, for each value of $a$.
Then, because of Lemma~\ref{turn}, each additional degenerate
solution would give rise to a branch of stable solutions, which could be followed using the parameter $c$,
to $c=0$. However, from Lemma~\ref{thmcz}, at $c=0$ there exists only one stable solution $(a,u_\dagger(a),0)$.
This would yield a contradiction.
\end{proof}

We restrict our attention to $\lambda_1<a<\lambda_2$. We observe that there are no degenerate solutions with
Morse index greater than zero. Otherwise, we would have that the integral
$
\int\left[|\nabla v|^2-av^2+f'(u)v^2\right]\,dx,
$
and hence
$
\int\left[|\nabla v|^2-av^2\right]\,dx
$,
is nonpositive on a two dimensional subspace of ${\cal H}$, which is not possible for $a<\lambda_2$.
This observation and the above results lead to
\begin{Thm}\label{stab}
 Let $\lambda_1<a<\lambda_2$. The set of solutions of\/ {\rm (\ref{a})} is
a compact connected one dimensional manifold in $\{a\}\times\h\times\R$.
There exist precisely two solutions for each $c\in\,]\csm,\csp[$, one stable, $(a,u^*_a(c),c)$,
 and the other nondegenerate with Morse index
 equal to one, $(a, u^\sharp_a(c),c)$. In addition, there exists exactly one
degenerate solution with Morse index equal to zero when $c=\csm$ and $c=\csp$.
\end{Thm}
In the next results, we consider the case $M>0$.
Recall the definitions of $\Lambda$ in (\ref{Lambda}) and of $T$ in (\ref{T}).
In parallel to Lemma~\ref{proj}, we can prove
\begin{Lem}
Suppose $M>0$. For $0<\hat t<T$, define
$$
\Lambda_{\hat t}=\Lambda \cap \{(t,c)\in \R^2 : \, \hat t\leq t\}\quad { \it and }\quad
\Lambda_{\hat t}^C=\{(t,c)\in \R^2: \, t\geq 0\}\setminus \Lambda_{\hat t}
$$
(see {\rm Figure~\ref{fig13}}).
There exists $\delta>0$ such that for all $\lambda_1<a<\lambda_1+\delta$ and $(a,u,c)$ solution of\/~{\rm (\ref{a})},
we have $(t,c)\in \Lambda_{\hat t}^C$.
\end{Lem}

\begin{figure}
\centering
\begin{psfrags}
\psfrag{c}{{\tiny $c$}}
\psfrag{b}{{\tiny $\tz$}}
\psfrag{t}{{\tiny $t$}}
\psfrag{s}{{\tiny $\Lambda_{\hat t}$}}
\psfrag{u}{{\tiny $\Lambda_{\hat t}^C$}}
\includegraphics[scale=.7]{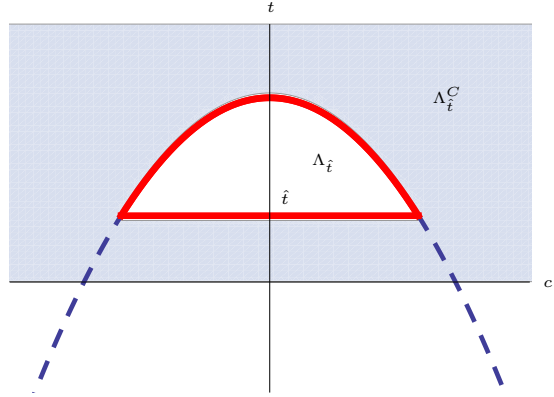}
\end{psfrags}
\caption{The regions $\Lambda_{\hat t}$ and $\Lambda_{\hat t}^C$.}\label{fig13}
\end{figure}

The proof is analogous to the one of Lemma~\ref{proj}, but here we use Proposition~\ref{ccota} to guarantee
that $c$ is bounded and the fact that the last term in (\ref{padepois}) is bounded,
since $(u_n)$ is uniformly bounded above, according to Remark~\ref{rmk}.

We examine the behavior of ${\cal D}_*^-$ and ${\cal D}_*^+$ as $a$ decreases to $\lambda_1$.
\begin{Prop}\label{dia}
 Suppose $M>0$. As $a$ decreases to $\lambda_1$, we have
\begin{eqnarray}
\lim_{a\searrow\lambda_1}(a,\usm(a),\csm(a))&=&(\lambda_1,0\,\phi+\cm\dlmu h, \cm),\label{cm}\\
\lim_{a\searrow\lambda_1}(a,\usp(a),\csp(a))&=&(\lambda_1,0\,\phi+\cp\dlmu h, \cp),\label{cp}
\end{eqnarray}
i.e.\
$$
\lim_{a\searrow\lambda_1}(\tsm(a),\csm(a))=(0,\cm),\quad \lim_{a\searrow\lambda_1}(\tsp(a),\csp(a))=(0,\cp),
$$
where $t_*^\pm(a):=\frac{\int u_*^\pm(a)\phi\,dx}{\int\phi^2\,dx}$ and $\cm$ and $\cp$ are given in\/ {\rm (\ref{cc})}.
\end{Prop}
\begin{proof}
 By Proposition~\ref{ccota}, $\csp(a)$ is bounded. By Lemma~\ref{continua}, we may assume, as $a\searrow\lambda_1$,
that $(a,\usp(a),\csp(a))$ converges,
say to $(\lambda_1,\usp(\lambda_1),\csp(\lambda_1))$, a solution of (\ref{a}). From equality (\ref{sinalt}),
$\tsp(\lambda_1):=\frac{\int \usp(\lambda_1)\phi\,dx}{\int\phi^2\,dx}$ is nonnegative. It is enough to prove that $\csp(\lambda_1)=\cp=\clp(0)$
because if $(t,\clp(0))\in\Lambda$ and $t\geq 0$, then $t=0$.
Since $\tsp(\lambda_1)\geq 0$ and $\csp$ is strictly decreasing, $\csp(\lambda_1)\leq \clp(0)$. Suppose, by contradiction, that $\delta=\clp(0)-\csp(\lambda_1)>0$
(see Figure~\ref{fig14}).
According to the definition of $\csp(\lambda_1)$, there exists $\eps>0$ such that for all $\lambda_1<a<\lambda_1+\eps$,
we have that $\csp(a)<\csp(\lambda_1)+\delta/2$. Lemma~\ref{turn} implies that for all $\lambda_1<a<\lambda_1+\eps$ and $(a,u,c)$
solution of (\ref{a}), $c\leq\csp(a)<\csp(\lambda_1)+\delta/2$. Now choose $(t_0,c_0)\in \partial\Lambda$, with $t_0>0$ and
$c_0>\clp(0)-\delta/2$. Applying Lemma~\ref{TFI} at $(t_0,c_0)$, we obtain  solutions in a neighborhood of $(t_0,c_0)$
corresponding to values of $a$ close to $\lambda_1$. We reach a contradiction to $\csp(a)<\csp(\lambda_1)+\delta/2$
for $a$ in a right neighborhood of $\lambda_1$. This proves $\csp(\lambda_1)=\clp(0)$.
\end{proof}
\begin{figure}
\centering
\begin{psfrags}
\psfrag{c}{{\tiny $c$}}
\psfrag{e}{{\tiny $\ (t_*^+(\lambda_1),\csp(\lambda_1))$}}
\psfrag{b}{{\tiny $(t,\csp(\lambda_1)+\delta/2)$}}
\psfrag{d}{{\tiny $=(t,\clp(0)-\delta/2)$}}
\psfrag{t}{{\tiny $t$}}
\psfrag{T}{{\tiny $(0,\clp(0))$}}
\psfrag{Q}{{\tiny $(t_0,c_0$)}}
\includegraphics[scale=.65]{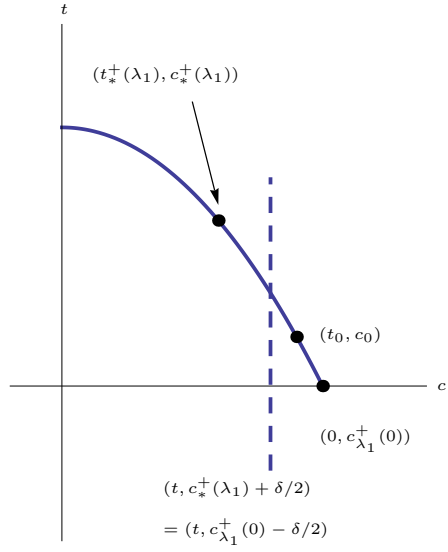}
\end{psfrags}
\caption{Illustration for the proof of Proposition~\ref{dia}.}\label{fig14}
\end{figure}

To conclude the analysis of the case $M>0$,
 we state a result on uniform convergence of the curves of solutions,
for fixed $\lambda_1<a<\lambda_2$,
 as $a$ decreases to $\lambda_1$. We take into account that the set of values $c$ for which there exists a solution, $[\csm(a),\csp(a)]$, in general
depends on $a$. Of course, due to (\ref{cm}) and (\ref{cp}),
for each $\cm<c<\cp$ there exist solutions $(a,u,c)$ for that value of $c$ and for $a$ in a right neighborhood of $\lambda_1$,
while for each $c<\cm$ and $c>\cp$ there do not exist solutions $(a,u,c)$ for that value of $c$ and for $a$ in a right neighborhood of $\lambda_1$.
\begin{Thm}\label{convl1}
Suppose $M>0$ and let $T$ be as in\/ {\rm (\ref{T})}. When $\cm\leq c\leq\cp$, define $t_c$ by
$$
t_c=\left\{\begin{array}{ll}
(\clm)^{-1}(c)&{\rm if }\ c<\clm(T),\\
T&{\rm if }\ \clm(T)\leq c\leq\clp(T),\\
(\clp)^{-1}(c)&{\rm if }\ c>\clp(T).
          \end{array}\right.
$$
For all $\delta>0$, there exists $\eps>0$ satisfying for all $\lambda_1<a<\lambda_1+\eps$ if the value
$c$ is such that there exists a solution $(a,u_a(c),c)$ then, in the case $c<\cm$ we have
\begin{eqnarray*}
 \|u_a(c)-(0\,\phi+\cm\dlmu h)\|_{\h}&<&\delta,
\end{eqnarray*}
and in the case $c>\cp$ we have
\begin{eqnarray*}
 \|u_a(c)-(0\,\phi+\cp\dlmu h)\|_{\h}&<&\delta;
\end{eqnarray*}
 if the value $c$ is such that there exist
solutions $(a,u^*_a(c),c)$ and $(a,u^\sharp_a(c),c)$
then, in the case $\cm\leq c\leq\cp$, we have
\begin{eqnarray*}
 \|u^*_a(c)-(t_c\phi+c\dlmu h)\|_{\h}&<&\delta,\\
 \|u^\sharp_a(c)-(0\,\phi+c\dlmu h)\|_{\h}&<&\delta.
\end{eqnarray*}
\end{Thm}
The proof is similar to the one of Theorem~\ref{conv}.

In Figure~\ref{fig16}, we plot the curves $c\mapsto(a,t(c),c)$ for $a$ in an interval $]\lambda_1,\lambda_1+\delta[$, for some
small $\delta>0$ and $M>0$.
\begin{figure}
\centering
\begin{psfrags}
\psfrag{c}{{\tiny $c$}}
\psfrag{a}{{\tiny $a$}}
\psfrag{t}{{\tiny $t$}}
\includegraphics[scale=.6]{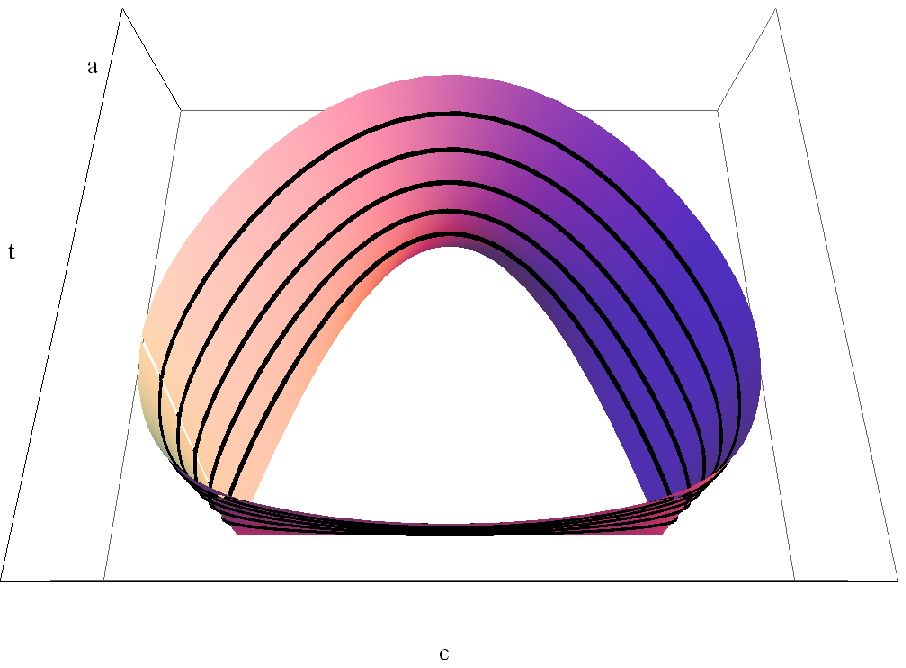}\ \ \ \ \ \ \ \
\includegraphics[scale=.6]{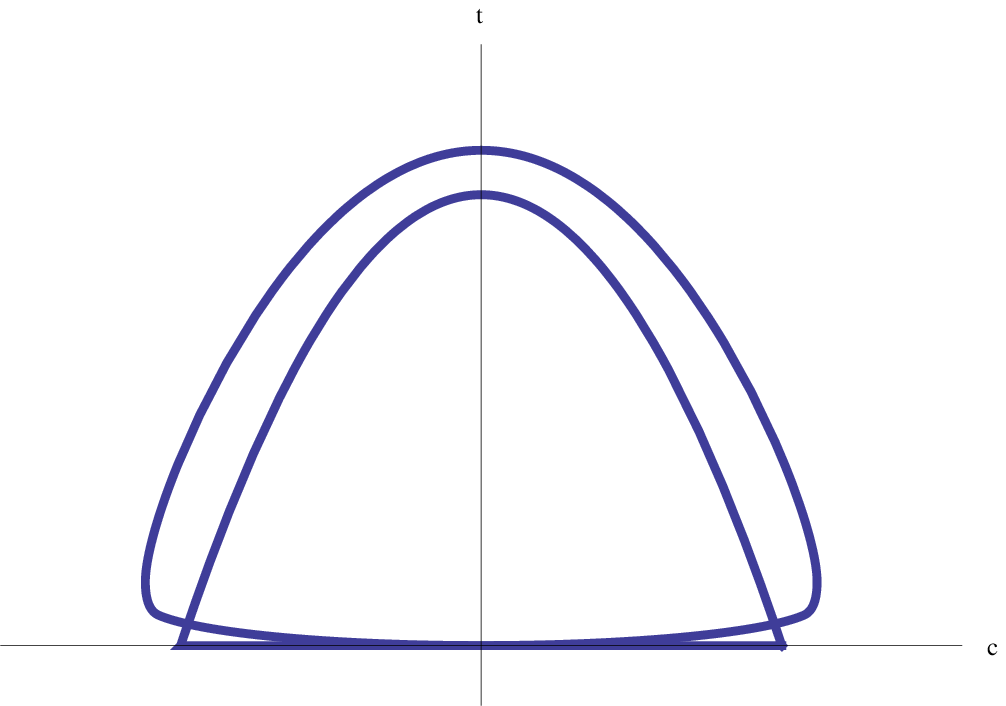}
\end{psfrags}
\caption{In the left image one can see the surface formed by the curves $c\mapsto(a,t(c),c)$ for $a$ in an interval $]\lambda_1,\lambda_1+\delta[$.
In the right image one can see a curve $c\mapsto(a,t(c),c)$ for an $a$ in $]\lambda_1,\lambda_1+\delta[$ and one can see
the boundary of the set $\Lambda\cap\{(t,c):t\geq 0\}$. This figure corresponds to a case where $M>0$.}\label{fig16}
\end{figure}

To conclude this section we consider the case $M=0$.
\begin{Thm}\label{convl1b}
 Suppose $M=0$. For all $\delta>0$, there exists $\eps>0$ satisfying for all $\lambda_1<a<\lambda_1+\eps$ if the value
$c$ is such that there exists a solution $(a,u_a(c),c)$, then
$$\|u_a(c)\|_{\h}<\delta.
$$
\end{Thm}
\section{Linear growth $a$ greater than or equal to $\lambda_2$}

We recall the assumption that the second eigenvalue of the Dirichlet Laplacian on $\Omega$
is simple and we call $\psi$ an associated eigenfunction, normalized so $\max_\Omega\psi=1$.
We define
$$
\beta=-\min_\Omega\psi,
$$
 so that $\beta>0$. To fix ideas, without loss of generality, we suppose
\begin{equation}\label{hpsi}
\int h\psi\,dx<0.
\end{equation}

For $a=\lambda_2$, the set of solutions of (\ref{a}) is completely described by
\begin{Thm}
\label{thmatl2} Suppose $f$ satisfies {\rm\bf (i)}-{\rm\bf (iv)} and $h$ satisfies {\rm\bf (a)}-{\rm\bf (c)}.
 Fix $a=\lambda_2$.
 The set of solutions $(\lambda_2, u, c)$ of {\rm (\ref{a})} is a compact connected one dimensional manifold
${\cal M}$ of class $C^1$
in $\{\lambda_2\}\times\h\times\R$. We have
$${\cal M}={\cal M}^\flat\cup{\cal L}\cup{\cal M}^\sharp\cup\{{\bm p}_*^+\}\cup{\cal M}^*\cup\{{\bm p}_*^-\},$$
where ${\cal L}$ connects ${\cal M}^\flat$ and ${\cal M}^\sharp$, $\{{\bm p}_*^+\}$ connects
${\cal M}^\sharp$ and ${\cal M}^*$, and, finally,\/ $\{{\bm p}_*^-\}$ connects
${\cal M}^*$ and ${\cal M}^\flat$.
Here
\begin{itemize}
\item
${\cal M}^\flat$ is a manifold of nondegenerate solutions with Morse index equal to one, which
is a graph $\{(\lambda_2,u^\flat_{\lambda_2}(c),c):c\in\,]\csm(\lambda_2),0[\}$.
\item
${\cal L}$ is a segment (a point in the case $M=0$) of degenerate solutions with Morse index
equal to one, $\bigl\{(\lambda_2,t\psi,0):t\in\bigl[-\frac{M}{\beta},M\bigr]\bigr\}$.
\item
${\cal M}^\sharp$ is a manifold of nondegenerate solutions with Morse index equal to one, which
is a graph $\{(\lambda_2,u^\sharp_{\lambda_2}(c),c):c\in\,]0,\csp(\lambda_2)[\}$.
\item ${\bm p}_*^+=(\lambda_2, \usp(\lambda_2), \csp(\lambda_2))$ is a degenerate solution with Morse index equal to zero.
\item ${\cal M}^*$ is a manifold of stable solutions, which
is a graph $\{(\lambda_2, u^*_{\lambda_2}(c),c):c\in\,]\csm(\lambda_2),\csp(\lambda_2)[\}$.
\item ${\bm p}_*^-=(\lambda_2, \usm(\lambda_2), \csm(\lambda_2))$ is a degenerate solution with Morse index equal to zero.
\end{itemize}
\end{Thm}
Theorem~\ref{thmatl2} is illustrated in Figure~\ref{fig1}.
\begin{figure}
\centering
\begin{psfrags}
\psfrag{c}{{\tiny $c$}}
\psfrag{u}{{\tiny $u$}}
\psfrag{z}{{\tiny ${\bm p}_*^+$}}
\psfrag{w}{{\tiny ${\bm p}_*^-$}}
\psfrag{s}{{\tiny ${\cal M}^\sharp$}}
\psfrag{t}{{\tiny ${\cal M}^*$}}
\psfrag{f}{{\tiny ${\cal M}^\flat$}}
\psfrag{l}{{\tiny ${\cal L}$}}
\psfrag{b}{{\tiny $(\usm(\lambda_2),\csm(\lambda_2))$}}
\psfrag{a}{{\tiny $(\usp(\lambda_2),\csp(\lambda_2))$}}
\includegraphics[scale=.6]{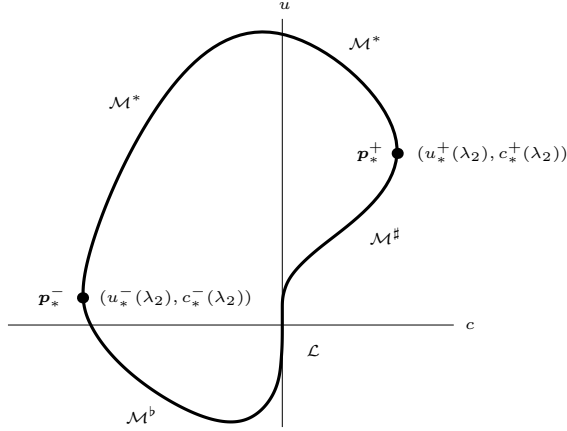}
\end{psfrags}
\caption{A bifurcation curve for $a=\lambda_2$.}\label{fig1}
\end{figure}
\begin{proof}[Sketch of the proof] We start at $(\lambda_2, u_\dagger(\lambda_2),0)$,
the stable solution at $c=0$. We may use the Implicit Function Theorem to follow the solutions for
 $c\in ]\csm(\lambda_2),\csp(\lambda_2)[$, arriving at the left at ${\bm p}_*^-$ and at the right at ${\bm p}_*^+$.
Arguing as in the proof of Theorem~1.2 of \cite{PG}, ${\bm p}_*^+$ is connected successively to ${\cal M}^\sharp$, ${\cal L}$ and ${\cal M}^\flat$.
By Lemma~6.1 in \cite{PG}, we may use the parameter $c$ to follow the branch ${\cal M}^\flat$ until we reach a degenerate
solution with Morse index equal to zero, as all solutions with Morse index equal to one are nondegenerate,
except for $(\lambda_2, t\psi,0)$ for $t\in\bigl[-\,\frac M\beta, M\bigr]$. The degenerate solution with Morse index equal to zero
must be ${\bm p}_*^-$, since it is the only one corresponding to a negative value of $c$. The branches ${\cal M}^\flat$
and ${\cal M}^*$ connect at ${\bm p}_*^-$. This is in accordance to Lemma~\ref{turn}.
\end{proof}

When $a>\lambda_2$, the set of solutions of (\ref{a}) is characterized in
\begin{Thm}
\label{thm} Suppose $f$ satisfies {\rm\bf (i)}-{\rm\bf (iv)}
and $h$ satisfies {\rm\bf (a)}-{\rm\bf (c)}. Without loss of generality, suppose {\rm (\ref{hpsi})} is true.
There exists $\delta>0$ such that the following holds.
Fix $\lambda_2<a<\lambda_2+\delta$.
The set of solutions $(a,u,c)$ of~{\rm (\ref{a})} is a compact connected one dimensional manifold ${\cal M}$ of class $C^1$
in $\{a\}\times\h\times\R$. We have ${\cal M}$ is the disjoint union
$${\cal M}={\cal M}^\flat\cup\{{\bm p}_\flat\}\cup{\cal M}^\natural\cup\{{\bm p}_\sharp\}\cup{\cal M}^\sharp
\cup\{{\bm p}_*^+\}\cup{\cal M}^*\cup\{{\bm p}_*^-\},$$
where $\{{\bm p}_\flat\}$ connects ${\cal M}^\flat$ and ${\cal M}^\natural$, $\{{\bm p}_\sharp\}$ connects
${\cal M}^\natural$ and ${\cal M}^\sharp$, $\{{\bm p}_*^+\}$ connects ${\cal M}^\sharp$ and ${\cal M}^*$,
and $\{{\bm p}_*^-\}$ connects ${\cal M}^*$ and ${\cal M}^\flat$.
Here
\begin{itemize}
\item ${\cal M}^\flat$ is a manifold
of nondegenerate solutions with Morse index equal to one, which
is a graph $\{(a,u^\flat_a(c),c):c\in\,]\csm(a),c_\flat(a)[\}$.
\item
${\bm p}_\flat=(a,u_\flat(a),c_\flat(a))$ is a degenerate solution with Morse index equal to one.
\item ${\cal M}^\natural$ is a manifold
of solutions with Morse index equal to one or to two,
$$\bigl\{(a,u^\natural_a(t),c^\natural_a(t)):u^\natural_a(t)=t\psi+\y^\natural_a(t),\ t\in J\bigr\},$$
with $c^\natural_a:J\to\R$,
$y^\natural_a:J\to\bigl\{y\in\h:\int y\psi\,dx=0\bigr\}$ and
$J=\bigl]-\frac{M}{\beta}-\eps_\flat,M+\eps_\sharp\bigr[$, for some $\eps_\flat,\eps_\sharp>0$.
\item
${\bm p}_\sharp=(a,u_\sharp(a),c_\sharp(a))$ is a degenerate solution with Morse index equal to one.
\item ${\cal M}^\sharp$ is a manifold
of nondegenerate solutions with Morse index equal to one, which
is a graph $\{(a,u^\sharp_a(c),c):c\in\,]c_\sharp(a),\csp(a)[\}$.
\item
${\bm p}_*^+=(a,\usp(a),\csp(a))$ is a degenerate solution with Morse index equal to zero.
\item ${\cal M}^*$ is the manifold of stable solutions, which
is a graph $\{(a,u^*_a(c),c):c\in\,]\csm,(a),\csp(a)[\}$.
\item
${\bm p}_*^-=(a,\usm(a),\csm(a))$ is a degenerate solution with Morse index equal to zero.
\end{itemize}
We have $(\cc^\natural_a)'(0)<0$ and
\begin{eqnarray*}
\lim_{t\searrow -\frac{M}{\beta}-\eps_\flat}(a,u^\natural_a(t),c^\natural_a(t))& =& (a,u_\flat(a),c_\flat(a)),\\
\lim_{t\nearrow M+\eps_\sharp}(a,u^\natural_a(t),c^\natural_a(t))&=& (a,u_\sharp(a),c_\sharp(a)).
\end{eqnarray*}
In particular, if $|c|$ is sufficiently small,
then {\rm (\ref{a})} has at least four solutions.
\end{Thm}
Theorem~\ref{thm} is illustrated in Figure~\ref{fig2}.
\begin{figure}
\centering
\begin{psfrags}
\psfrag{c}{{\tiny $c$}}
\psfrag{u}{{\tiny $u$}}
\psfrag{b}{{\tiny $(\usm(a),\csm(a))$}}
\psfrag{a}{{\tiny $(\usp(a),\csp(a))$}}
\psfrag{d}{{\tiny $(u_\flat(a),c_\flat(a))$}}
\psfrag{e}{{\tiny $(u_\sharp(a),c_\sharp(a))$}}
\psfrag{k}{{\tiny $\,{\bm p}_\flat$}}
\psfrag{j}{{\tiny ${\bm p}_\sharp$}}
\psfrag{z}{{\tiny ${\bm p}_*^+$}}
\psfrag{w}{{\tiny ${\bm p}_*^-$}}
\psfrag{s}{{\tiny ${\cal M}^\sharp$}}
\psfrag{t}{{\tiny ${\cal M}^*$}}
\psfrag{f}{{\tiny ${\cal M}^\flat$}}
\psfrag{l}{{\tiny ${\cal M}^\natural$}}
\includegraphics[scale=.6]{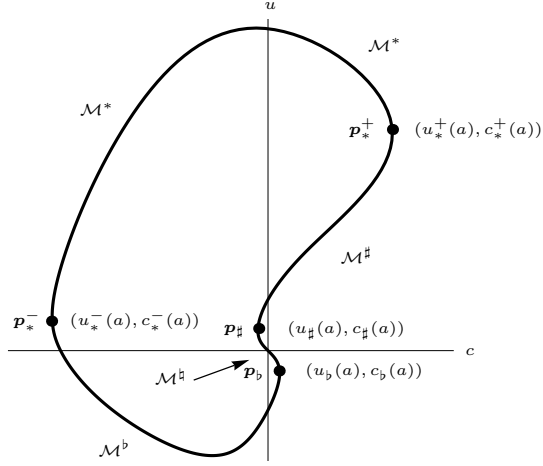}
\end{psfrags}
\caption{A bifurcation curve for $\lambda_2<a<\lambda_2+\delta$.}\label{fig2}
\end{figure}
\begin{proof}[Sketch of the proof]
Using Lemma~7.3 of \cite{PG},
we introduce a chart $]\lambda_2-\eps,\lambda_2+\eps[\times
\bigl]-\,\frac{M}{\beta}-\tilde{\eps},M+\tilde{\eps}\bigr[$ in the $(a,t)$ plane, around $(\lambda_2,0)$, to describe the solutions of
(\ref{a}) in a neighborhood of $(\lambda_2,0,0)$.
Suppose $\lambda_2<a<\lambda_2+\delta$.
 We start at $(a,0)$ and vary the parameter $t$ in $\bigl]-\,\frac{M}{\beta}-\tilde{\eps},M+\tilde{\eps}\bigr[$
to follow the solutions of (\ref{a}).
By choosing $\delta$ small enough, we can guarantee that when $t$ reaches $-\,\frac{M}{\beta}-\tilde{\eps}$ and
$M+\tilde{\eps}$ we can switch to the parameter $c$ to follow the solutions of (\ref{a}).
Indeed, consider the curve ${\cal D}_\varsigma$ of degenerate solutions with Morse index equal to one,
constructed in Lemma~7.2 of \cite{PG}.
A small enough choice of $\delta$ will make
 the projection of ${\cal D}_\varsigma$ in the chart not intersect either $\bigl(a,-\,\frac{M}{\beta}-\tilde{\eps}\bigr)$
or $(a,M+\tilde{\eps})$. Arguing as in the proof of Proposition~7.5 of \cite{PG},
we know that the solutions with coordinates $\bigl(a,-\,\frac{M}{\beta}-\tilde{\eps}\bigr)$
and $(a,M+\tilde{\eps})$ are nondegenerate and have Morse index equal to one. We also know that
when we arrive at the solution with coordinates
$(a,M+\tilde{\eps})$ we have to
increase $c$, and
when we arrive at the solution with coordinates
$\bigl(a,-\frac{M}{\beta}-\tilde{\eps}\bigr)$ we have to
decrease $c$, to follow the solutions out of the chart. By further reducing $\delta$, if necessary, Lemma~7.4 of \cite{PG}
together with Proposition~\ref{ccota},
assure that there are no degenerate solutions with Morse index equal to one when we use the parameter $c$ to
follow the solutions outside the chart.
If we find a degenerate solution it will necessarily have Morse index equal to zero.
But we know these lie on $\D_*^-$ and $\D_*^+$. So one can finish by arguing
as in the proof of Theorem~\ref{thmatl2}.
\end{proof}

We finish with a word on qualitative properties of solutions.
A nodal domain of a function $u$ is a connected component of $\Omega\setminus u^{-1}(0)$.
There are several works relating the number of nodal domains 
of solutions of elliptic equations with their Morse indices.
Possibly the first result in this direction was the Courant Nodal Domain Theorem~\cite{CH}, which states that number of nodal domains
of an $n$-th eigenfunction of the Dirichlet Laplacian on $\Omega$ is less than or equal to $n$.
But Courant gave examples where the $n$-th eigenfunction only has two nodal domains.

In the one dimensional case, $\Omega\subset\R^N$ with $N=1$, Sturm's Comparison Theorems allow one to establish that 
for some {\em linear homogeneous}\/ equations the number of nodal domains of a solution is
equal to its Morse index (see \cite[Chapter~8, Theorem~2.1]{CL}).

Courant's Nodal Domain Theorem was also generalized to some superlinear elliptic equations (for example, see
\cite{BL},
\cite{BW} and \cite{Y}).
Let us recall, in general terms, how the main argument goes.
Suppose $u\in H^1_0(\Omega)$ is a weak solution of
\begin{equation}\label{last}-\Delta u=g(x,u)\ \ \ {\rm in}\ \Omega,
\end{equation}
where $g$ satisfies appropriate growth conditions, $g(x,0)=0$ for all $x\in\Omega$, and $g'(x,t)t^2>g(x,t)t$ for all $x\in\Omega$ and $t\neq 0$
(where the prime denotes the derivative with respect to $t$).
Let $v=u\chi$, where $\chi$ is the characteristic function of a nodal domain of $u$. Using
\cite[Lemma~1]{MP}, the function $v$ belongs to $H^1_0(\Omega)$. Multiplying both sides of \eqref{last} by $v$ and integrating
over $\Omega$, and using the superlinear assumption on $g$, one obtains
\begin{eqnarray*}
\int|\nabla v|^2\,dx&=&\int g(x,v)v\,dx<\int g'(x,v)v^2\,dx.
\end{eqnarray*}
Thus, the functional 
$$
v\mapsto \int|\nabla v|^2\,dx-\int g'(x,u)v^2\,dx
$$ is negative in the direction $v$. Furthermore, two different nodal domains are associated to two independent functions $v$, with disjoint
supports. 
So the Morse index of $u$ is greater than or equal to the number of nodal domains of $u$.

The type of argument in the previous paragraph does not apply to  (\ref{a}) where
$$
g(x,u)=au-f(u)-ch(x)
$$
is not superlinear and
does not vanish at $u=0$, unless $c=0$. Moreover,
using specific examples, in the simplest one dimensional case, one can check that a stable a solution of~(\ref{a}) may have two nodal domains.
Hence, the usual relation between the number of nodal domains and the Morse index does not hold for our equation.

\end{document}